%% file: ulln_v2.tex
\documentclass[11pt]{article}

 \usepackage{multirow}
 \usepackage[table,xcdraw,dvipsnames]{xcolor}

\usepackage[utf8]{inputenc}
\input{fast-draft-macro}

\usepackage{hhline}

\usepackage{pifont}
\newcommand{\cmark}{\ding{51}}
\newcommand{\xmark}{\ding{55}}

\title{Failure of uniform laws of large numbers \\ for subdifferentials and beyond}

\author{Lai~Tian\thanks{Daniel J.~Epstein Department of Industrial and Systems Engineering, University of Southern California, Los Angeles, CA 90089; Emails: {\tt \{laitian, royset\}@usc.edu}} \and Johannes O.~Royset\footnotemark[1]}

\date{March 16, 2026}

\usepackage{fullpage}

\usepackage{multirow}

\def\epsilon{\varepsilon}

\newcommand{\bit}{\mathop{\textnormal{bit}}\nolimits}
\newcommand{\ball}{\mathbb{B}}
\newcommand{\compl}{\mathsf{c}}

\usepackage{bbm}
\usepackage{dsfont}

\usepackage{enumitem}

\begin{document}
\maketitle
\begin{abstract}
We provide counterexamples showing that uniform laws of large numbers do not hold for subdifferentials under natural assumptions. Our constructions are univariate random Lipschitz functions and bivariate random convex functions with two smooth pieces. Consequently, they resolve the questions posed by Shapiro and Xu [\textit{J.~Math.~Anal.~Appl.}, 325 (2007), 1390--1399] in the negative. They also demonstrate the failure of certain graphical and pointwise laws for subdifferentials, revealing fundamental barriers to the consistency of sample-average approximation and subdifferential approximation.
\end{abstract}

\section{Introduction}\label{sec:intro}

Uniform laws of large numbers (LLNs) for gradients are of fundamental importance in stochastic programming \cite[Section 7.2.5]{shapiro2021lectures}, and their validity and rates of convergence are now well understood for random smooth functions; see, e.g., \cite{mei2018landscape,foster2018uniform}. However, the question of whether similar uniform laws hold for their nonsmooth counterparts has remained open for nearly two decades \cite[Remark 2]{shapiro2007uniform}. Positive results are known only for special function classes \cite{shapiro1989asymptotic,shapiro2007uniform,teran2008uniform,xu2010uniform,ruan2024subgradient}, for \emph{enlarged} subdifferentials \cite{shapiro1996convergence,shapiro2007uniform}, and for weaker notions of convergence \cite{norkin-wets,davis2022graphical,salim2023strong}; see also the discussion in \cite[Section 7.2.6]{shapiro2021lectures} and \cite[p.~471]{molchanov2017theory}.

To set the stage, let nonempty $X \subset \reals^d$ be compact, $\Xi \subset \reals^m$, and $\bm{\xi},  \bm{\xi}^1,\bm{\xi}^2,\ldots$ be independent and identically distributed (iid) $\Xi$-valued random variables on the complete probability space $(\Omega, \mathcal{F}, \mathbb{P})$. Throughout, we use boldface for random variables. Given  $f:\Xi \times \reals^d \to \reals$ with (Borel) measurable $f(\cdot, x)$ for all $x\in \reals^d$, assume that the expectation function $x \mapsto \Ex[f(\bm{\xi},x)]$ is finite-valued and $f(\xi,\cdot)$ is $L(\xi)$-Lipschitz on an open neighborhood of $X$ for any $\xi \in \Xi$, with measurable $L$ and $\Ex[L(\bm{\xi})] < \infty$. A general uniform LLN for \emph{enlarged} set-valued mappings appears in \cite[Theorem 2]{shapiro2007uniform} and can be applied to the Clarke subdifferential $\partial_x f(\xi,x)=\partial (f(\xi,\cdot))(x)$.
 Specifically, for any fixed $r > 0$, $\mathbb{P}$-almost surely ($\mathbb{P}$-a.s.), one has
\begin{align}
\lim_{\nu \to \infty}\sup_{x \in X} \exs\Big(\tfrac{1}{\nu}\nsum_{i=1}^\nu \partial_x f (\bm{\xi}^i, x); \bigcup\nolimits_{y\in \ball(x, r)\cap X}\Ex[\partial_x f(\bm{\xi},y)]\Big) =0, \label{eq:enlarged-a}\\
\lim_{\nu \to \infty}\sup_{x \in X} \exs\Big(\Ex[\partial_x f(\bm{\xi},x)]; \bigcup\nolimits_{y\in \ball(x, r)\cap X}\tfrac{1}{\nu}\nsum_{i=1}^\nu \partial_x f(\bm{\xi}^i, y)\Big) =0, \label{eq:enlarged-b}
\end{align}
where $\exs(\cdot;\cdot)$ is the \emph{excess} of a set relative to another one and the integral is taken in the Aumann sense; see the end of this section for a summary of notation. 
When $f(\xi,\cdot)$ is further assumed to be subdifferentially regular \cite[Definition 7.25]{VaAn} on $X$ for all $\xi \in \Xi$, we have 
\[
\Ex[\partial_x f(\bm{\xi},x)]=\partial \Ex[ f(\bm{\xi},\cdot)](x),\quad \tfrac{1}{\nu}\nsum_{i=1}^\nu \partial_x f (\xi^i, x)=\partial \big(\tfrac{1}{\nu}\nsum_{i=1}^\nu f (\xi^i, \cdot)\big)(x),
\] 
for any $\xi^1,\ldots, \xi^\nu \in \Xi$ and $x \in X$;
see \cite[Theorem 2.7.2]{clarke1990optimization}.

It is then natural to ask, as also explicitly mentioned in \cite[Remark 2]{shapiro2007uniform}, whether we can set $r=0$ in \eqref{eq:enlarged-a} and/or \eqref{eq:enlarged-b}, which, if both are true, would imply the following uniform law for subdifferentials in the Hausdorff sense:
\begin{equation}\label{eq:r0ulln}
\lim_{\nu \to \infty}\sup_{x \in X} \setd\Big(\Ex[ \partial_x f(\bm{\xi},x)], \tfrac{1}{\nu}\nsum_{i=1}^\nu \partial_x f (\bm{\xi}^i, x) \Big) =0.
\end{equation}
While the general questions about random Lipschitz functions remain elusive, the uniform law in \eqref{eq:r0ulln} is actually achievable for several special function classes. One sufficient condition is that, for every $x \in X$, the subdifferential $\partial_x f(\xi,x)$ is a singleton for almost every $\xi \in \Xi$, which  implies \eqref{eq:r0ulln} and also that  $x\mapsto \Ex[f(\bm{\xi},x)]$ is continuously differentiable; see \cite[Proposition 2.2]{shapiro1989asymptotic}, \cite[Theorem 6]{shapiro2007uniform}, and \cite[Theorem 7.52]{shapiro2021lectures}.
Beyond requiring the expectation function to be smooth, the work \cite{teran2008uniform} guarantees \eqref{eq:r0ulln} when the subdifferential mapping $\partial_x f$ satisfies a separable range condition; see \cite[Theorem 4]{teran2008uniform} and also \cite[Theorem 5.1.31]{molchanov2017theory}.
While this may be appealing for certain applications, the range of $\partial_x f$ is not separable even for the simple random convex function $f(\bm{\xi},x)=\max\{x - \bm{\xi},0\}$ with $X=[0,1]$ and $\bm{\xi}$ uniformly distributed on $\Xi=[0,1]$ as pointed out in \cite[Example 3.4]{norkin-wets}; see also \cite[Example 5.1.33]{molchanov2017theory}.
The recent work \cite{ruan2024subgradient} establishes the uniform law \eqref{eq:r0ulln} for a subclass of random convex-composite functions. The uniform law in \cite[Theorem 5]{ruan2024subgradient} requires the outer function to be convex, univariate, and deterministic, while the inner function can be random but smooth and Vapnik--Chervonenkis-major (VC-major) \cite[Section 2.6.4]{vanderVaartWellner.23}. It remains unclear whether the VC-type assumption can be removed when the explicit rate in \cite[Theorem 5]{ruan2024subgradient} is not required.
Resolving these questions would deepen our understanding of sample-average approximation (SAA) from stochastic programming and empirical risk minimization in machine learning.

In this paper, we give negative answers to all these questions by providing explicit counterexamples.
Our negative results hold under natural assumptions, requiring only simple nonsmooth components. 
 We report the main results in \Cref{sec:lip,sec:cvx}, and then discuss in \Cref{sec:dis} the implications of our constructions for the consistency of SAA and for subdifferential approximation.
 All deferred proofs  are collected in \Cref{sec:proof}.

\paragraph{Notation.} 
We use mostly standard notation. 
The Borel $\sigma$-algebra on $\Xi$ is denoted by $\mathcal{B}(\Xi)$. 
 A random function $f:\Xi\times \reals^d \to \reals$ is \emph{Carath\'eodory} if $f(\xi, \cdot)$ is continuous for each $\xi$ and $f(\cdot, x)$ is (Borel) measurable for each $x$; whenever defined, its \emph{gradient} and \emph{Clarke subdifferential} at $y$ are denoted by $\nabla_x f(\xi,y)=\nabla (f(\xi,\cdot))(y)$ and $\partial_x f(\xi,y)=\partial (f(\xi,\cdot))(y)$, respectively. Let $\|\cdot\|$ denote the  \emph{Euclidean norm}, and define $\ball(x,r)=\{y\mid \|y-x\|\leq r\}$. For $\nu \in \nats=\{1,2,\ldots\}$, write $[\nu]=\{1, \ldots, \nu\}$.
 A function $g:\reals^d \to \reals$ is $L$-Lipschitz on $X$ (or simply $L$-Lipschitz when $X=\reals^d$) if $|g(x) - g(y)| \leq L\|x - y\|$ for all $x,y \in X$; it is $L$-smooth on $X$ if it is continuously differentiable ($C^1$) with $L$-Lipschitz $\nabla g$ on $X$ (or on an open set containing $X$). A set-valued mapping $S:\Xi \rightrightarrows \reals^d$ is measurable if $\{\xi \in \Xi \mid S(\xi)\cap C \neq \emptyset\} \in \mathcal{B}(\Xi)$ for every open $C\subset \reals^d$; its expectation $\Ex[S(\bm{\xi})]\subset \reals^d$ is defined in the Aumann sense; see, e.g., \cite[Definition 7.39]{shapiro2021lectures}.
 For nonempty sets $C,D\subset \reals^d$ and a point $x \in \reals^d$, the distance between $x$ and $C$ is $\dist(x,C)=\inf_{z \in C} \|x - z\|$; the \emph{excess} of $C$ relative to $D$ is 
$
\exs(C;D)=\nsup_{z \in C} \dist(z, D);
$ 
and the \emph{Hausdorff distance} between $C$ and $D$ is $\setd(C,D) = \max\{\exs(C;D), \exs(D;C)\}$. 
The $\{0,1\}$-indicator function of a set $C$ is defined by $\mathds{1}_C(x)=1$ if $x \in C$ and $\mathds{1}_C(x)=0$ otherwise.
 For $\epsilon \geq 0$, the $\epsilon$-subdifferential of a convex function $h:\reals^d \to \reals$ is defined as
	$
	\partial^\epsilon h(x)=\{s \in \reals^d\mid h(y) - h(x) \geq s^\top (y - x) - \epsilon,\forall y\},
	$
	which 
	coincides with the usual convex subdifferential when $\epsilon=0$. For $t \in \reals$, we write $\lceil t\rceil$ and $\lfloor t\rfloor$ for the ceiling and floor of $t$, respectively.

\section{Main Results}\label{sec:main}

\subsection{Random Lipschitz Functions}\label{sec:lip}

We begin with a general negative result for univariate random Lipschitz functions.

\begin{theorem}[random Lipschitz]\label{thm:hard-lip}
Let $X=[0,1]\subset \reals$ and $\bm{\xi}, \bm{\xi}^1,\bm{\xi}^2,\ldots$ be iid random variables on the complete probability space $(\Omega,\mathcal{F},\mathbb{P})$, each uniformly distributed on $\Xi=[0,1]$.	
There exist a Carath\'eodory function $f:\Xi\times \reals \to \reals$ and $\delta^\nu \in (0,1]\downarrow 0$ such that the following hold.
\begin{enumerate}[label=\textnormal{(\alph*)}]
	\item For any $\xi \in \Xi$, the function $f(\xi, \cdot)$ is $1$-Lipschitz.
	\item For any $x \in \reals$, $|\Ex[f(\bm{\xi},x)]|<\infty$.
	\item For any $\nu \in \nats$, the set $D^\nu=\{x  \in X\mid f(\xi,\cdot)\text{ is }C^1\text{ on }\ball(x,\delta^\nu), \forall \xi \in \Xi\}$ is nonempty. \label{item:thm:lip-D}
	\item $\mathbb{P}$-a.s., one has \label{item:thm:lip-ineqs}
\[
\liminf_{\nu \to \infty} \sup_{x \in D^\nu} \inf_{y,\hat{y} \in \ball(x,\delta^\nu)} \Big| \Ex[\nabla_x f(\bm{\xi},y)] -\tfrac{1}{\nu}\nsum_{i=1}^\nu\nabla_x f(\bm{\xi}^i,\hat{y}) \Big| \geq \tfrac{1}{2}.
\] 
\end{enumerate}
\end{theorem}

The definition of the set $D^\nu$ requires that, for every $x \in D^\nu$, the function $f(\xi,\cdot)$ is $C^1$ on $\ball(x,\delta^\nu)$ for all $\xi \in \Xi$, so that $\nabla_x f(\xi,\cdot)$ is well defined (hence $\nabla_x f(\cdot, y)$ is measurable) and continuous on $\ball(x,\delta^\nu)$. This is a stringent requirement, and the fact that we can show $D^\nu \neq \emptyset$ in \ref{item:thm:lip-D} and restrict $x$ to $D^\nu$ in \ref{item:thm:lip-ineqs} only strengthens the theorem. The following corollary is immediate from \Cref{thm:hard-lip}.

\begin{corollary}[subdifferential]\label{coro:lip}
	Under the assumptions of \Cref{thm:hard-lip}, and for $X$, $\Xi$, $\{\bm{\xi}, \bm{\xi}^\nu\}_\nu$, $f$, and $\{\delta^\nu\}_\nu$ constructed there, the following hold.  
	\begin{enumerate}[label=\textnormal{(\alph*)}]
	\item For any $\xi \in \Xi$, $\partial_x f(\xi,\cdot)$ is outer semicontinuous with $\sup_{z \in \partial_x f(\xi,x)} \|z\| \leq 1$ for all $x \in \reals$.
	\item For any $x \in \reals$, $\partial_x f(\cdot,x)$ is measurable.
	\item For any sequence $\{r^\nu \in [0, \delta^\nu]\}_\nu$, $\mathbb{P}$-a.s., one has \label{item:coro:lip-c}
	\begin{align}
\liminf_{\nu \to \infty}\sup_{x \in X} \exs\Big(\tfrac{1}{\nu}\nsum_{i=1}^\nu \partial_x f(\bm{\xi}^i, x); \bigcup\nolimits_{y\in \ball(x, r^\nu)\cap X}\Ex[\partial_x f(\bm{\xi},y)]\Big) \geq \tfrac{1}{2}, \label{eq:coro:subd-a}\\
\liminf_{\nu \to \infty}\sup_{x \in X} \exs\Big(\Ex[\partial_x f(\bm{\xi},x)]; \bigcup\nolimits_{y\in \ball(x, r^\nu)\cap X}\tfrac{1}{\nu}\nsum_{i=1}^\nu \partial_x f(\bm{\xi}^i, y)\Big) \geq \tfrac{1}{2}. \label{eq:coro:subd-b}%
\end{align}
	\end{enumerate}
\end{corollary}
\begin{proof}
	The outer semicontinuity, boundedness, and measurability are from standard arguments in the beginning of \cite[Section 3]{shapiro2007uniform}. For measurability, see also \cite[Lemma 4]{norkin1986stochastic}. Part \ref{item:coro:lip-c} follows from the facts that $\partial_x f(\xi, \cdot) = \{\nabla_x f(\xi, \cdot)\}$ on $\ball(x,\delta^\nu)$ for any $x \in D^\nu \subset X$ and \Cref{thm:hard-lip}.
\end{proof}

In \Cref{coro:lip}, we show that neither \eqref{eq:enlarged-a} nor \eqref{eq:enlarged-b} can hold when $r=0$, so \eqref{eq:r0ulln} is in general impossible for random Lipschitz functions. Consequently, this gives negative answers to the questions posed in \cite[Remark 2]{shapiro2007uniform}. Moreover, it shows that the lower bounds in \eqref{eq:coro:subd-a} and \eqref{eq:coro:subd-b} remain valid even for $r^\nu \downarrow 0$ shrinking sufficiently fast, so that a trivial ``smoothing'' technique cannot circumvent our negative result.

\begin{remark}\label{rmk}
It is worth noting that the restriction to the subset $D^\nu \subset X$ in \Cref{thm:hard-lip} yields applications far beyond those for the Clarke subdifferential in \Cref{coro:lip}. For any $\xi \in \Xi$, the function $f(\xi,\cdot)$ is actually smooth on $D^\nu$, and most standard notions of subdifferential coincide there; see, e.g., \cite[Chapter 8]{VaAn}. Even for the Clarke subdifferential, we may interchange the sum/expectation and the subdifferential operator, replacing $\tfrac{1}{\nu}\nsum_{i=1}^\nu \partial_x f(\bm{\xi}^i,x)$ by $\partial\big(\tfrac{1}{\nu}\nsum_{i=1}^\nu f(\bm{\xi}^i,\cdot)\big)(x)$ in \eqref{eq:coro:subd-a} and $\Ex[\partial_x f(\bm{\xi},x)]$ by $\partial \Ex [f(\bm{\xi},\cdot)](x)$ in \eqref{eq:coro:subd-b}, while the lower bounds there still hold. 
Consequently, the failure of uniform laws is not an artifact of the particular subdifferential employed, but reflects a more intrinsic obstruction that extends to a broad class of local approximation concepts.
\end{remark}	
	
	\subsection{Random Convex Functions}\label{sec:cvx}

\Cref{sec:lip} shows that the uniform laws in \eqref{eq:enlarged-a} and \eqref{eq:enlarged-b} with $r=0$ cannot hold for random Lipschitz functions. 
One might speculate that Lipschitz functions are somewhat ``wild'' and the situation would be better for other nonsmooth functions.
A natural class to examine next is that of random convex functions, since convexity is widely regarded as favorable. The following uniform law for $\epsilon$-subdifferentials is classical.

\begin{theorem}[cf.~{\cite[Proposition 3.4]{shapiro1996convergence} and \cite[Theorem 7.56]{shapiro2021lectures}}]\label{prop:cvx-eps} 
Let nonempty $X \subset \reals^d$ be compact, $\Xi \subset \reals^m$, and $\bm{\xi}, \bm{\xi}^1,\bm{\xi}^2,\ldots$ be $\Xi$-valued iid random variables on the complete probability space $(\Omega,\mathcal{F},\mathbb{P})$. Let $f:\Xi \times \reals^d \to \reals$ be a Carath\'eodory function such that $f(\xi,\cdot)$ is convex for any $\xi \in \Xi$. Assume that $|\Ex[f(\bm{\xi},x)]|<\infty$ for any $x \in \reals^d$. Then, for any fixed $\epsilon > 0$, $\mathbb{P}$-a.s., one has
\begin{equation}\label{eq:cvx-subd}
\lim_{\nu \to \infty} \sup_{x \in X} \setd\Big(\partial^\epsilon\Ex[f(\bm{\xi},\cdot)](x), \partial^\epsilon\big(\tfrac{1}{\nu}\nsum_{i=1}^\nu f(\bm{\xi}^i,\cdot)\big)(x)\Big)=0.
\end{equation}
\end{theorem}

\Cref{prop:cvx-eps} differs in several ways from the enlarged uniform laws in \eqref{eq:enlarged-a} and \eqref{eq:enlarged-b}. First, there is no Lipschitz assumption in \Cref{prop:cvx-eps}. Second, for a convex real-valued $h$ and a point $x$, unlike the $r$-enlargement $\cup_{y \in \ball(x,r)} \partial h(y)$, the $\epsilon$-subdifferential $\partial^\epsilon h(x)$ is intrinsically a \emph{global} notion by definition. Third, the mapping $x \mapsto \partial^\epsilon h(x)$ with $\epsilon > 0$ is continuous with respect to the Hausdorff metric, whereas the mapping $x \mapsto \cup_{y \in \ball(x,r)} \partial h(y)$ may fail to be. Moreover, as emphasized in \cite[p.~382]{shapiro2021lectures}, for $\xi^1,\ldots,\xi^\nu \in \Xi$ and a function $f:\Xi \times \reals^d \to \reals$ such that $f(\xi,\cdot)$ is convex for every $\xi$, the $\epsilon$-subdifferential sum rule
\[
\partial^\epsilon \big(\tfrac{1}{\nu}\nsum_{i=1}^\nu f(\xi^i,\cdot)\big)(x)=\tfrac{1}{\nu}\nsum_{i=1}^\nu \partial_x^\epsilon f(\xi^i,x),
\]
where $\partial_x^\epsilon f(\xi^i,x)=\partial^\epsilon (f(\xi^i,\cdot))(x)$, 
holds for $\epsilon = 0$ but can fail for $\epsilon > 0$ and $\nu > 1$.

Therefore, for uniform laws of the form \eqref{eq:enlarged-a} and/or \eqref{eq:enlarged-b} (also \eqref{eq:cvx-subd}) for random convex functions, it is natural to ask whether one can take $r = 0$ (equivalently, $\epsilon = 0$ in \eqref{eq:cvx-subd}). We next see that the answer to this question crucially depends on the dimension $d$ of $x$.

\subsubsection*{Univariate Case ($d=1$)} 
Since the negative result in \Cref{thm:hard-lip} for random Lipschitz functions is based on a univariate construction, for random convex functions we likewise first consider the univariate case.

\begin{proposition}[univariate random convex]\label{prop:cvx-1dim}
For a Borel measurable $\Xi\subset \reals^m$, let $\bm{\xi}, \bm{\xi}^1,\bm{\xi}^2,\ldots$ be $\Xi$-valued iid random variables on the complete probability space $(\Omega,\mathcal{F},\mathbb{P})$. Let $f:\Xi \times \reals \to \reals$ be a Carath\'eodory function such that $f(\xi,\cdot)$ is convex and $L(\xi)$-Lipschitz for any $\xi \in \Xi$ with Borel measurable $L$ and $\Ex[L(\bm{\xi})] < \infty$. Assume that $|\Ex[f(\bm{\xi},x)]|<\infty$ for any $x \in \reals$. Then, $\mathbb{P}$-a.s., one has
\[
\lim_{\nu \to \infty} \sup_{x \in \reals}\setd\Big(\Ex[\partial_x f(\bm{\xi},x)], \tfrac{1}{\nu}\nsum_{i=1}^\nu\partial_x f(\bm{\xi}^i,x)\Big)=0.
\]
\end{proposition}

In sharp contrast to the univariate random Lipschitz case, convexity allows us to take $r = 0$ in both \eqref{eq:enlarged-a} and \eqref{eq:enlarged-b}, even without any compactness assumption. The uniform law in \Cref{prop:cvx-1dim} holds on the entire real line. In spirit, it is similar to the classical Glivenko--Cantelli theorem and, even when restricted to a compact $X$, already covers cases that were previously unknown.
For illustration, let $X = [0,3]$, and let $\bm{\xi} = (\bm{\xi}_1,\bm{\xi}_2)$, where $\bm{\xi}_1$ and $\bm{\xi}_2$ are independent random variables, with $\bm{\xi}_1$ uniformly distributed on $[0,1]$ and $\bm{\xi}_2=2$ almost surely. Define $f(\xi,x)=\max\{x-\xi_1,0\}+\max\{ \xi_2  - x,0\}$. It is easy to see that $x \mapsto \Ex[f(\bm{\xi},x)]$ is not differentiable (it has a kink at $x = 2$), and a computation similar to \cite[Example 3.4]{norkin-wets} shows that the range of $\partial_x f$ is not separable. Moreover, one can show that $f$ cannot be written in the univariate convex-composite form considered in \cite[Section 4]{ruan2024subgradient}. Hence, a subdifferential uniform law for $f$ cannot be deduced from the results in \cite{shapiro1989asymptotic,shapiro2007uniform,teran2008uniform,xu2010uniform,ruan2024subgradient}. However, \Cref{prop:cvx-1dim} applies.

\subsubsection*{General Case ($d\geq 2$)}
Perhaps surprisingly, the situation changes drastically as soon as $d = 2$.

\begin{theorem}[random convex]\label{thm:hard-cvx-comp}
Let $X=\{0\}\times [0,1]\subset \reals^2$, and $\bm{\xi}, \bm{\xi}^1,\bm{\xi}^2,\ldots$ be iid random variables on the complete probability space $(\Omega,\mathcal{F},\mathbb{P})$, each uniformly distributed on $\Xi = [0,1]$.	
There exist a Carath\'eodory function $g:\Xi\times \reals^2 \to \reals$ and $\delta^\nu \in (0,1]\downarrow 0$ such that for 
\[
(\xi,x)\mapsto f(\xi,x)=\max\{g(\xi,x),0\}+35\|x\|^2,
\] the following hold.
\begin{enumerate}[label=\textnormal{(\alph*)}]
	\item For any $\xi \in \Xi$, $g(\xi,\cdot)$ is $70$-Lipschitz, $70$-smooth, and $f(\xi,\cdot)$ is convex, $140$-Lipschitz on $X$.
	\item For any $x \in \reals^2$, $|\Ex[f(\bm{\xi},x)]|<\infty$.
	\item For any $\nu \in \nats$, the set $D^\nu=\{x  \in X\mid f(\xi,\cdot)\text{ is }C^1\text{ on }\ball(x,\delta^\nu), \forall \xi \in \Xi\}$ is nonempty.
	\item $\mathbb{P}$-a.s., one has
	\[
\liminf_{\nu \to \infty} \sup_{x \in D^\nu} \inf_{y,\hat{y} \in \ball(x,\delta^\nu)} \Big\| \Ex[\nabla_x f(\bm{\xi},y)] -\tfrac{1}{\nu}\nsum_{i=1}^\nu\nabla_x f(\bm{\xi}^i,\hat{y}) \Big\| \geq \tfrac{1}{2}.
\] 
\end{enumerate} 
\end{theorem}
 
In \Cref{thm:hard-cvx-comp}, we show that even in dimension $d=2$ there exists a random convex function $f:\Xi \times \reals^2 \to \reals$ for which a negative result analogous to \Cref{thm:hard-lip} holds. Therefore, in the same vein as \Cref{coro:lip}, one cannot take $r=0$ in \eqref{eq:enlarged-a} and \eqref{eq:enlarged-b}; equivalently, a uniform law with $\epsilon = 0$ in \eqref{eq:cvx-subd} cannot hold. 
The same lower bound extends to all $d\ge2$ by padding with zero coordinates, and to any compact superset of $X$. Thus, restricting to a lower-dimensional set $X$ strengthens, rather than weakens, the negative result.

However, \Cref{thm:hard-cvx-comp} says more than these, since the random convex $f$ there is highly structured. Let us make the following observations: 
\begin{itemize}
	\item In contrast to the construction used to prove \Cref{thm:hard-lip} (see \Cref{sec:proof-lip}), which  involves countably infinitely many pieces\footnote{Here, by the number of ``pieces,'' we mean the smallest number of connected regions partitioning the domain such that the function is smooth on the interior of each region.} (see \Cref{fig:lip}), the function $f$ in \Cref{thm:hard-cvx-comp} is a random convex function with only two pieces (see \Cref{fig:cvx}). Indeed, the only source of nonsmoothness in the construction is the function $t\mapsto \max\{t,0\}$. Moreover, by \cite[Theorem 7.52]{shapiro2021lectures}, we have 
	\[
	\lim_{\nu \to \infty} \sup_{x \in X} \Big\| \Ex[\nabla_x g(\bm{\xi},x)] - \tfrac{1}{\nu}\nsum_{i=1}^\nu\nabla_x g(\bm{\xi}^i,x) \Big\| = 0
	\]
	for the function $g$ emerging from \Cref{thm:hard-cvx-comp}, $\mathbb{P}$-a.s., with $X$ being any nonempty compact set.
	Thus, each smooth component appearing in the definition of $f$ satisfies a uniform law on its own. However, when they are combined through this (arguably) simplest nonsmooth operation $t \mapsto \max\{t,0\}$, the uniform law fails. This phenomenon is driven by the oscillating boundary between the connected sets $\{x \mid g(\xi,x) > 0\}$ and $\{x \mid g(\xi,x) < 0\}$, whose intersection with $X = \{0\} \times [0,1]$ gives rise to countably infinitely many kinks in the univariate function $t \mapsto \max\{g(\xi,(0,t)),0\}$, leading to an oscillating ``active set.''
	\item The negative results in \Cref{thm:hard-cvx-comp} hold for any function class that contains $f$. An interesting example is a subclass of the random convex-composite functions studied in \cite[Section 4]{ruan2024subgradient}. For $\kappa < \infty$, consider a convex, \emph{piecewise affine}, $\kappa$-Lipschitz  $h:\reals \to \reals$ and a Carath\'eodory  $c:\Xi \times \reals^d \to \reals$ such that $c(\xi,\cdot)$ is $\kappa$-Lipschitz and $\kappa$-smooth for each $\xi$. Since $\kappa$ is finite and $h$ is piecewise affine, the constants in \cite[Assumptions C.1--C.3]{ruan2024subgradient} are \emph{distribution-free}. By \cite[Theorem 5]{ruan2024subgradient}, for large $\nu$, with (outer) probability at least $1-\delta$, one has 
	\[
\sup_{x \in B}\setd\Big(\Ex[\partial_x(h\circ c)(\bm{\xi},x)], \tfrac{1}{\nu}\nsum_{i=1}^\nu\partial_x (h\circ c)(\bm{\xi}^i,x)\Big) \leq O\bigg(\sqrt{\frac{d + \textsf{VC}(\mathcal{C})\log \nu +\log(1/\delta)}{\nu} } \bigg),
\]
	where $B$ is an open Euclidean ball, $\mathcal{C}=\{\{\xi \mid c(\xi,x) \geq t\} \mid x \in \reals^d, t \in \reals \}$, and $\textsf{VC}(\mathcal{C})$ denotes the VC dimension of $\mathcal{C}$; see \cite[Section 2.6.1]{vanderVaartWellner.23}. This upper bound is also \emph{distribution-free}, but is non-trivial only if $\textsf{VC}(\mathcal{C}) < \infty$; in other words, the function class $\{c(\cdot,x) \mid x \in \reals^d\}$ is VC-major; see \cite[Section 2.6.4, Exercise 2.6.13]{vanderVaartWellner.23}. The VC-type assumptions are usually employed to obtain \emph{distribution-free} results \cite[Section 2.8.1]{vanderVaartWellner.23} and may not be necessary for universal Glivenko--Cantelli-type results, let alone $\mathbb{P}\circ\bm{\xi}^{-1}$-Glivenko--Cantelli-type results such as \cite[Theorem 7.52]{shapiro2021lectures}. Hence, it is natural to ask whether the VC-major assumption here can be dispensed with when uniformity over the underlying distributions and an explicit convergence rate are not required.
	Consider the convex piecewise affine $h = \max\{\cdot, 0\}$ and let $c=g$, where $g$ is as in \Cref{thm:hard-cvx-comp}. We can write the function $f$ in \Cref{thm:hard-cvx-comp} as
	\[
	f(\xi,x) = h(c(\xi,x)) + 35 \|x\|^2.
	\]
	Ignoring the deterministic term $x \mapsto 35\|x\|^2$, this is exactly a random convex-composite function of the above form with $\kappa=70$ but $\textsf{VC}(\mathcal{C})=\infty$.\footnote{To see it, let $\bit_k:[0,1]\to \{0,1\}$ and $g:\Xi\times \reals^2 \to \reals$ be defined in \Cref{sec:bit,sec:proof-cvx}.
When $x=(0,1/k)\in X$ and $t=0$, one has $\{\xi\mid g(\xi,x)\geq t\}=\{\xi\mid \bit_k(\xi)=1\}$. To show $\textsf{VC}(\mathcal{C})=\infty$, it suffices to show that, for any $n\in\nats$, there exist $\xi^1,\ldots,\xi^n\in \Xi$ such that, for any $b\in\{0,1\}^n$, one has $\bit_k(\xi^i)=b_i$ for each $i\in[n]$ and some $k\in\nats$. Let $\{b^k\mid k\in[2^n]\}=\{0,1\}^n$ cycle through all $2^n$ patterns. A valid construction is $\xi^i=\sum_{k=1}^{2^n} 2^{-k}b^k_i$ for each $i\in[n]$.} When $X$ in \Cref{thm:hard-cvx-comp} is a subset of $B$, we see that, even without requiring uniformity over distributions, the VC-major assumption cannot, in general, be dropped without causing the uniform law to fail.
	\item 
	Another example is median regression. Let $\bm{\xi} = (\bm{\xi}_1,\bm{\xi}_2)$, where $\bm{\xi}_1$ and $\bm{\xi}_2$ are $\Xi$-valued and $[-1,1]$-valued random variables, respectively, with $\Xi \subset \reals^m$. Let $\psi:\Xi \times \reals^d \to \reals$ be a Carath\'eodory function such that $\psi(\xi_1,\cdot)$ is Lipschitz and smooth for each $\xi_1 \in \Xi$. Median regression corresponds to the function  $f(\xi,x)=|\xi_2-\psi(\xi_1,x)|$. It is easy to see that the construction in \Cref{thm:hard-cvx-comp} can be rewritten in this form, since $f(\xi,x)=\max\{2\xi_2 - 2\psi(\xi_1,x),0\} - \xi_2 + \psi(\xi_1,x)$ with the last term being handled by \cite[Theorem 7.52]{shapiro2021lectures}. Therefore, in general, uniform laws fail for the subdifferential of $\Ex[f(\bm{\xi},\cdot)]$ when $d > 1$.

\end{itemize}

It is straightforward to verify that the conclusion of \Cref{coro:lip} still holds under the assumptions of \Cref{thm:hard-cvx-comp}, and we omit the details for brevity. We end this subsection with a negative result for $\epsilon$-subdifferentials, complementing the positive result in \Cref{prop:cvx-eps} and addressing the case of rapidly shrinking $\epsilon^\nu \downarrow 0$.
\begin{corollary}[$\epsilon$-subdifferential]\label{coro:cvx}
	Under the assumptions of \Cref{thm:hard-cvx-comp}, let $X$, $\Xi$, $\{\bm{\xi}, \bm{\xi}^\nu\}_\nu$, $f$, and $\{\delta^\nu\}_\nu$ be constructed there. For any sequence $\{\epsilon^\nu \in [0, (\delta^\nu)^2]\}_\nu$,  $\mathbb{P}$-a.s., one has
	\[
	\begin{aligned}
\liminf_{\nu \to \infty}\sup_{x \in X} \exs\Big(\partial^{\epsilon^\nu}\big(\tfrac{1}{\nu}\nsum_{i=1}^\nu  f(\bm{\xi}^i, \cdot)\big)(x); \partial^{\epsilon^\nu}\Ex[ f(\bm{\xi},\cdot)](x)\Big) \geq \tfrac{1}{2}, \\
\liminf_{\nu \to \infty}\sup_{x \in X} \exs\Big(\partial^{\epsilon^\nu}\Ex[ f(\bm{\xi},\cdot)](x); \partial^{\epsilon^\nu}\big(\tfrac{1}{\nu}\nsum_{i=1}^\nu f(\bm{\xi}^i, \cdot)\big)(x)\Big) \geq \tfrac{1}{2}.
\end{aligned}
\]

\end{corollary}
\begin{proof}
For any $x^\nu \in D^\nu$ in \Cref{thm:hard-cvx-comp}, let $g_1^\nu \in \partial^{\epsilon^\nu}\Ex[ f(\bm{\xi},\cdot)](x^\nu)$ and $g_2^\nu \in \partial^{\epsilon^\nu}\big(\tfrac{1}{\nu}\nsum_{i=1}^\nu f(\bm{\xi}^i, \cdot)\big)(x^\nu)$ be arbitrary. By Br{\o}ndsted--Rockafellar \cite[p.~608]{brondsted1965subdifferentiability} and $\epsilon^\nu \in [0,(\delta^\nu)^2]$, we get $\|g_1^\nu - \Ex[\nabla_x f(\bm{\xi},y^\nu_1)]\| \leq \delta^\nu$ and $\|g_2^\nu - \tfrac{1}{\nu}\sum_{i=1}^\nu \nabla_x f(\bm{\xi}^i,y^\nu_2)\| \leq \delta^\nu$, where $y_1^\nu,y_2^\nu \in \ball(x^\nu, \delta^\nu)$. Then, $\|g_1^\nu - g_2^\nu\|\geq \|\Ex[\nabla_x f(\bm{\xi},y^\nu_1)] - \tfrac{1}{\nu}\sum_{i=1}^\nu \nabla_x f(\bm{\xi}^i,y^\nu_2)\| - 2\delta^\nu$.
This completes the proof by invoking \Cref{thm:hard-cvx-comp} and $\delta^\nu \downarrow 0$.
\end{proof}

\subsection{Discussion}\label{sec:dis}
The constructions used to prove \Cref{thm:hard-lip,thm:hard-cvx-comp} have implications beyond uniform LLNs.
In this subsection, we discuss several of these consequences, focusing on SAA consistency and subdifferential approximation; see \Cref{tab} for a summary of the results.

\begin{table}[]
\centering
\setlength{\arrayrulewidth}{0.8pt}
\begin{tabular}{c|c|c|c|c|c}
\textbf{functions} & \begin{tabular}[c]{@{}c@{}}$\Ex \partial$ and $\partial \Ex$\\ uniform\end{tabular} & \begin{tabular}[c]{@{}c@{}}$\Ex \partial$\\ graphical\end{tabular} & \begin{tabular}[c]{@{}c@{}}$\partial  \Ex$\\ graphical\end{tabular} & \begin{tabular}[c]{@{}c@{}}$\Ex \partial$\\ pointwise\end{tabular} & \begin{tabular}[c]{@{}c@{}}$\partial \Ex$\\ pointwise\end{tabular} \\ \hline
Lipschitz          & \cellcolor[HTML]{FFFFC7}\xmark{} \Cref{coro:lip}                                  &                                                                    & \cellcolor[HTML]{FFFFC7}\xmark{} \Cref{cor:hard-ghp}                &                                                                    & \cellcolor[HTML]{FFFFC7}\xmark{} \Cref{cor:hard-point}             \\ \hhline{-|-|~|-|~|-}
Lip.~\& regular    & \cellcolor[HTML]{FFFFC7}\xmark{} \Cref{coro:cvx}                                  & \multirow{-2}{*}{\cmark{} {\cite{shapiro2007uniform,norkin-wets}}}              & \cmark{} {\cite{shapiro2007uniform,norkin-wets}}                                & \multirow{-2}{*}{\cmark{} {\cite{artstein1975strong}}}              & \cmark{}  {\cite{artstein1975strong}}                              
\end{tabular}
\caption{Validity and failure of various laws of large numbers for subdifferentials. Uniform laws fail regardless of whether the expectation is taken before or after the subdifferential operator (see \Cref{rmk}), in contrast to graphical and pointwise laws; see \Cref{sec:dis}.}
\label{tab}
\end{table}

\paragraph{Consistency of SAA.}
The preceding subsections show that, for certain seemingly benign nonsmooth functions, a uniform LLN for subdifferentials cannot hold. This, however, does not preclude consistency of stationary points of the SAA objective, provided stationarity is interpreted in a weaker sense. Specifically, let a random vector $\bm{x}^\nu$ be a \emph{weak stationary} point of the SAA objective $\tfrac{1}{\nu}\sum_{i=1}^\nu f(\bm{\xi}^i,\cdot)$ $\mathbb{P}$-a.s.;  i.e., $0 \in \tfrac{1}{\nu}\sum_{i=1}^\nu \partial_x f(\bm{\xi}^i,\bm{x}^\nu)$. The qualifier ``weak'' reflects that, in general, 
\[
\partial \big(\tfrac{1}{\nu}\nsum_{i=1}^\nu f (\xi^i, \cdot)\big)(x)\subset \tfrac{1}{\nu}\nsum_{i=1}^\nu \partial_x f (\xi^i, x), \quad 
\partial \Ex[ f(\bm{\xi},\cdot)](x)\subset \Ex[\partial_x f(\bm{\xi},x)], 
\] 
 so every \emph{stationary} point (i.e., $x$ such that $0\in \partial (\tfrac{1}{\nu}\nsum_{i=1}^\nu f (\xi^i, \cdot))(x)$) is weak stationary, but not conversely. In the SAA setting of \Cref{sec:intro}, if weak stationary points $\bm{x}^\nu \to x$ $\mathbb{P}$-a.s., then $0 \in \Ex[\partial_x f(\bm{\xi},x)]$, i.e., $x$ is a weak stationary point of $\Ex[f(\bm{\xi},\cdot)]$; see, e.g., \cite[Theorem~4.2]{xu2010uniform}. 
 
 Weak stationarity is used by many in stochastic programming (see, e.g., \cite{xu2010uniform,teran2010consistency,burke2020subdifferential}), whereas  stationarity is seldom used without assuming subdifferential regularity; see, e.g., \cite{shapiro2007uniform,milz2025consistency}. One may ask whether this is merely a matter of technique or reflects a genuine obstruction. The following corollary of our construction in \Cref{sec:proof-lip} shows it is the latter.

\begin{corollary}[consistency]\label{cor:hard-ghp}
Under the assumptions of \Cref{thm:hard-lip}, and for $X$ and $\{\bm{\xi}, \bm{\xi}^\nu\}_\nu$ constructed there, there exist a Carath\'eodory function $h:\Xi\times \reals \to \reals$ and random points $\{\bm{x}^\nu:\Omega \to X\}_\nu$ such that the following hold.
\begin{enumerate}[label=\textnormal{(\alph*)}]
	\item $\mathbb{P}$-a.s., for all sufficiently large $\nu \in \nats$, $\bm{x}^\nu$ is a local minimizer of $\tfrac{1}{\nu}\nsum_{i=1}^\nu   h(\bm{\xi}^i,\cdot)$.
	\item $\mathbb{P}$-a.s., one has $\bm{x}^\nu \to x \in X$, but $0 \not\in \partial \Ex[h(\bm{\xi},\cdot)](x)$.
\end{enumerate}
\end{corollary}

\Cref{cor:hard-ghp} shows that, as $\nu \to \infty$, cluster points of local minimizers of the SAA objective need not even be stationary. This strengthens the well-known observation that local optimality itself may fail to be preserved under SAA limits; see \cite[(4.6)]{bastin2006convergence}. By the Fermat rule, a local minimizer is stationary. Therefore, \Cref{cor:hard-ghp} highlights the breakdown of a \emph{graphical} LLN for subdifferentials when one considers $\partial \Ex[f(\bm{\xi},\cdot)]$ rather than $\Ex[\partial_x f(\bm{\xi},\cdot)]$; cf.~\cite[Example 4.4]{norkin-wets}. 

\paragraph{Subdifferential Approximation.}
The subdifferential of an expectation function can be approximated in several ways, e.g., via smoothing \cite{burke2020subdifferential}. Here we focus on the SAA approach (see, e.g., \cite{birge1995subdifferential}), which is also tied to the Artstein--Vitale LLN for set-valued mappings \cite{artstein1975strong}. 
Compared with the uniform and graphical laws discussed above, the Artstein--Vitale LLN can be viewed as a \emph{pointwise} law. In the SAA setting of \Cref{sec:intro}, it implies (see also \cite[Theorem~7.53]{shapiro2021lectures}) that,
\[ \lim_{\nu \to \infty} \setd \Big(\tfrac{1}{\nu}\nsum_{i=1}^\nu \partial_x f(\bm{\xi}^i, x), \Ex[ \partial_x f(\bm{\xi}, x)] \Big) =0. 
\]
As in the graphical laws of \cite{shapiro2007uniform,norkin-wets}, the above approximation is in the \emph{weak} sense: it controls $\tfrac{1}{\nu}\sum_{i=1}^\nu \partial_x f(\bm{\xi}^i,x)$ versus $\Ex[\partial_x f(\bm{\xi},x)]$, rather than $\partial(\tfrac{1}{\nu}\sum_{i=1}^\nu f(\bm{\xi}^i,\cdot))(x)$ versus $\partial \Ex[f(\bm{\xi},\cdot)](x)$. It is thus natural to ask whether the latter type of convergence can ever hold in general. The following corollary of our construction in \Cref{sec:proof-lip} gives a negative answer.
\begin{corollary}[subdifferential approximation]\label{cor:hard-point}
Under the assumptions of \Cref{thm:hard-lip}, and for $X$,  $\{\bm{\xi}, \bm{\xi}^\nu\}_\nu$, and $f$ constructed there, there exists $x \in X$ such that $\mathbb{P}$-a.s., one has 
\[
\liminf_{\nu \to \infty} \exs \Big(\partial \big(\tfrac{1}{\nu}\nsum_{i=1}^\nu f(\bm{\xi}^i, \cdot) \big)(x);
\partial \Ex[f(\bm{\xi}, \cdot)](x) \Big) \geq \tfrac{1}{2}.
\]
\end{corollary}

Notably, the special pointwise case treated in \Cref{cor:hard-point} already suffices to rule out a variant of \eqref{eq:r0ulln} when the expectation is taken \emph{inside} the subdifferential operator. It does not, however, rule out the version with the expectation \emph{outside}, since the corresponding pointwise statement always holds by the Artstein--Vitale LLN. By contrast, \Cref{thm:hard-lip} shows that the uniform law \eqref{eq:r0ulln} fails regardless of whether the expectation is placed inside or outside the subdifferential; see \Cref{rmk}.
 
\paragraph{DC Construction.}
Analogues of \Cref{cor:hard-ghp,cor:hard-point} also hold for a difference-of-convex (DC) variant of the construction used to prove \Cref{thm:hard-cvx-comp}; see \Cref{sec:dc} for details. 
Hence, restricting attention to DC functions with finitely many smooth pieces does not circumvent our negative results.
In particular, the \emph{sufficient separation condition} cannot be removed in \cite[Corollary 5.1]{qi2022asymptotic} while preserving consistency of \emph{d-stationary} points.

\section{Proofs}\label{sec:proof}
We now present the proofs of our main results, beginning with a random bits gadget.
\subsection{A Probabilistic Gadget}\label{sec:bit}
For any $\xi \in [0,1]$ and  $k \in \{0,1,\ldots\}$, let 
\[
\bit_k(\xi) = \lfloor 2^k \xi \rfloor - 2\lfloor 2^{k-1}\xi \rfloor  \in \{0,1\},
\]
which can be understood as the $k$th binary digit of  $\xi$ using the terminating-zeros convention for dyadic rationals.  For $\xi \in [0,1]$, we have the binary expansion
	$
	\xi = \nsum_{k=0}^\infty 2^{-k}\bit_k(\xi).
	$
	
Let $\bm{\xi}^1,\bm{\xi}^2,\ldots$ be iid random variables on the complete probability space $(\Omega,\mathcal{F},\mathbb{P})$, each with  uniform distribution on $\Xi = [0,1]$.
	It is classical that the random variables $\{\bit_k(\bm{\xi}^i)\}_{i,k \in \nats}$ are iid Bernoulli random variables with parameter $\tfrac{1}{2}$; see, e.g., \cite[Section 4.6]{williams1991probability} and \cite[Lemma 2.20]{kallenberg2021foundations}.

\begin{lemma}\label{prop:bit-measure}
For any $k \in \nats$, the function $\xi\mapsto\bit_k(\xi)$ is $\mathcal{B}(\Xi)$-measurable. 	
\end{lemma}
\begin{proof}
Simply observe that $t \mapsto \lfloor t \rfloor$ is $\mathcal{B}(\Xi)$-measurable.
\end{proof}

\begin{lemma}\label{prop:E}
There exist an $\mathcal{F}$-measurable set $E\subset \Omega$ with $\mathbb{P}(E)=1$ and $\nats$-valued random variables $\bar{\bm{\nu}},\bm{k}^1, \bm{k}^2,\ldots$ such that for any $\omega \in E$ and any $\nu \geq \bar{\bm{\nu}}(\omega)$, one has
\[
\nu \leq {\bm{k}^\nu(\omega)} \leq \lceil 2^{\nu+1}\ln(\nu+1)\rceil \quad\text{and}\quad
\bit_{{\bm{k}^\nu(\omega)}}(\bm{\xi}^1(\omega))=\cdots = \bit_{{\bm{k}^\nu(\omega)}}(\bm{\xi}^\nu(\omega)) = 1.
\] 
\end{lemma}
\begin{proof}
	Let $K^\nu=\lceil 2^{\nu+1}\ln(\nu+1)\rceil$. For any $\nu,k \in \nats$, define sets
	\[
		E^\nu_k =\bigcap\nolimits_{i \in [\nu]}\{\bit_k(\bm{\xi}^i)=1 \},\qquad E=\bigcup\nolimits_{\bar{\nu} \geq 1} \bigcap\nolimits_{\nu \geq \bar{\nu}}\bigcup\nolimits_{k=\nu}^{K^\nu} E^\nu_k,
	\] 
	which are $\mathcal{F}$-measurable from \Cref{prop:bit-measure}.
	By independence of $\{\bit_k(\bm{\xi}^i)\}_{i\in[\nu],k\in\nats}$, we have
	\[
	\begin{aligned}
	\mathbb{P}\Big(\bigcap\nolimits_{k=\nu}^{K^\nu} (E^\nu_k)^\compl\Big)&=\prod\nolimits_{k=\nu}^{K^\nu} \mathbb{P}((E^\nu_k)^\compl)= \prod\nolimits_{k=\nu}^{K^\nu} \Big(1-\prod\nolimits_{i =1}^\nu \mathbb{P}(\{\bit_k(\bm{\xi}^i)=1\}) \Big)\\
&= (1-2^{-\nu})^{K^\nu-\nu+1} \leq e^{-2^{-\nu}(K^\nu-\nu+1)} \leq \tfrac{\sqrt{e}}{(\nu+1)^2}.
	\end{aligned}
	\]
	Then, $\sum_{\nu=1}^\infty \mathbb{P}(\cap_{k =\nu}^{K^\nu} (E^\nu_k)^\compl) < \infty$. By Borel--Cantelli, we have 
	\[
	\mathbb{P}\Big(\bigcap\nolimits_{\bar{\nu} \geq 1} \bigcup\nolimits_{\nu \geq \bar{\nu}}\bigcap\nolimits_{k=\nu}^{K^\nu} (E^\nu_k)^\compl\Big)=\mathbb{P}(E^\compl)=0.
	\]
	For any $\nu \in \nats$, define $\bm{k}^\nu(\omega) = \nu$ if $\omega \notin \cup_{k=\nu}^{K^\nu} E_k^\nu$ and $\bar{\bm{\nu}}(\omega) = 1$ if $\omega \notin E$. 
	Otherwise, set
	\[
	\bm{k}^\nu(\omega) =
\min\{k \mid \nu \leq k \leq K^\nu, \omega \in  E_k^\nu\},
\quad
\bar{\bm{\nu}}(\omega) =
\min\{\bar{\nu} \geq 1 \mid \omega \in \cap_{\nu \geq \bar{\nu}} \cup_{k=\nu}^{K^\nu} E_k^\nu\},
	\]
which are $\mathcal{F}$-measurable and complete the proof by construction.
\end{proof}

\subsection{Proof of \Cref{thm:hard-lip}}\label{sec:proof-lip}

\paragraph{Intuition.} 
Our proof is inspired by a construction that refutes one-sided uniform consistency of epigraphical LLNs; see \cite[Example 4.1]{artstein1995consistency}. The key idea is to accumulate countably infinitely  many affine pieces near the origin, with the slope of each piece determined by a random bit, $\bit_k(\bm{\xi}^i) \in \{0,1\}$. 
Then, using \Cref{prop:E}, we can almost surely find a random point in $D^\nu$ at which  the functions $\{f(\bm{\xi}^i,\cdot)\mid i \in [\nu]\}$ are all $C^1$ with gradient consistently equal to $1$, thereby exhibiting a constant gap between the empirical and the average cases. Our proof is divided into five steps.

\begin{figure}[t]
	\centering
  \includegraphics[width=0.59\textwidth]{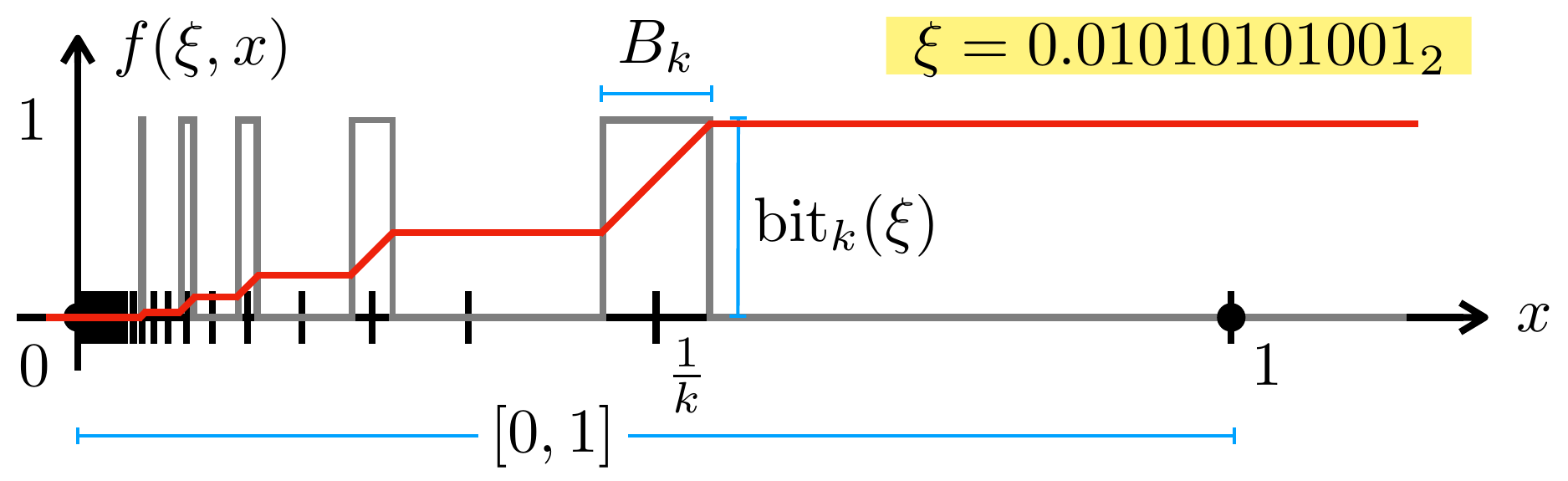}
  \caption{The function $f(\xi,\cdot)$ in \Cref{sec:proof-lip} when $\xi=0.01010101001_2$.}\label{fig:lip}
\end{figure}

\paragraph{Step 1.}
Define open sets $B_k=\{x \in \reals \mid |x-\tfrac{1}{k}|< r_k\}$ with $r_k=\frac{1}{4k^2}$ for $k \in \nats$.
	Let 
	\[
	g(\xi,x)=
\nsum_{k=1}^\infty \mathds{1}_{B_k}(x)\bit_k(\xi).
	\]
	Since $\tfrac{1}{2}(\frac{1}{k}-\frac{1}{k+1})\geq \frac{1}{4k^2}=r_k$, the sets $\{B_k\}_{k}$ are disjoint, so that  $g(\xi,x) \in \{0,1\}$.
	Let
	\[
	f(\xi,x) = \int_0^{\max\{x,0\}} g(\xi,t) \dd t,
	\]
	which is well-defined as explained below; see \Cref{fig:lip}.
		Define $\delta^\nu=(8\lceil 2^{\nu+1}\ln(\nu+1)\rceil^2)^{-1}.$
\paragraph{Step 2.} The $\{0,1\}$-valued function $g(\xi,\cdot)$ is a sum of measurable indicator functions, then also measurable.
Since $g(\xi,\cdot)$ is bounded and has only countably many points of discontinuity, the Lebesgue criterion implies that $g(\xi,\cdot)$ is Riemann integrable on compact sets, and hence $f$ is well defined.
For any $x,y \in \reals$, one has 
$|f(\xi,x) - f(\xi,y)|\leq |\int_{\min\{x,y\}}^{\max\{x,y\}} g(\xi,t)\dd t|\leq |x-y|$, hence $f(\xi, \cdot)$ is $1$-Lipschitz continuous.
From \Cref{prop:bit-measure}, the function $\xi \mapsto \bit_k(\xi)$ is measurable. By Tonelli, we have $f(\xi,x)=\sum_{k=1}^\infty \bit_k(\xi)\int_0^{\max\{x,0\}}\mathds{1}_{B_k}(t)\dd t<\infty$, hence $f(\cdot,x)$ is also measurable.
	Therefore, $f$ is Carath\'eodory. Moreover, for any $x \in \reals$, $|\Ex[f(\bm{\xi},x)]| \leq \int_0^{\max\{x,0\}}1\dd t<\infty$.
	
	\paragraph{Step 3.}

	Let $p_k = \tfrac{1}{k}$ and $\Delta_k=\frac{r_k}{2}=\tfrac{1}{8k^2}$. For any $k \in \nats$, $\xi \in \Xi$, and $y \in \ball(p_k, \Delta_k)\subset B_k$, one has  
	\[
	g(\xi,y)=\bit_k(\xi), \qquad f(\xi,y)=\bit_k(\xi)(y - p_k+\Delta_k)+\int_0^{p_k - \Delta_k} g(\xi,t)\dd t .
	\]
	Hence, the function 
	$f(\xi,\cdot)$ is smooth on $\cup_k\ball(p_k,\Delta_k)$. Moreover, for any $k \in \nats$ and $y \in \ball(p_k,\Delta_k)$, we can write $
	\nabla_xf(\xi,y) = g(\xi,y)= \bit_k(\xi)$.
	
\paragraph{Step 4.} 
Let $\mathcal{F}$-measurable $E\subset \Omega$ with $\mathbb{P}(E)=1$, $\bar{\bm{\nu}}$, and $\{\bm{k}^\nu\}_\nu$ be defined in \Cref{prop:E}.
	For any $\omega \in E$ and $\nu \geq \bar{\bm{\nu}}(\omega)$, one has $\bit_{\bm{k}^\nu(\omega)}(\bm{\xi}^i(\omega))=1$ for all $i \in [\nu]$ and some $\nu \leq \bm{k}^\nu(\omega) \leq \lceil 2^{\nu+1}\ln(\nu+1)\rceil$. 
	 Let $\bm{p}^\nu(\omega)=p_{\bm{k}^\nu(\omega)} \in \{p_k\}_{k }$. For any $y \in \ball(\bm{p}^\nu(\omega), \Delta_{\bm{k}^\nu(\omega)})$, we have
	\[
	\tfrac{1}{\nu}\nsum_{i=1}^\nu\nabla_x f(\bm{\xi}^i(\omega),y)= \tfrac{1}{\nu}\nsum_{i=1}^\nu \bit_{\bm{k}^\nu(\omega)}(\bm{\xi}^i(\omega))= 1.
	\]
	Meanwhile, for any $k \in \nats$ and $y \in \ball(p_k, \Delta_k)$, deterministically, we have 
	\[
	\Ex[\nabla_xf(\bm{\xi},y)] = \Ex_{\bm{\xi}}[\bit_k(\bm{\xi})] = \tfrac{1}{2}.
	\]
\paragraph{Step 5.} Note that 
\[
0\leq \delta^\nu
=
\big(8\lceil 2^{\nu+1}\ln(\nu+1)\rceil^2\big)^{-1}
= \Delta_{\lceil 2^{\nu+1}\ln(\nu+1)\rceil}
\leq 
\Delta_{\bm{k}^\nu(\omega)}.
\]
Therefore, for any $\omega \in E$,  we conclude that $\bm{p}^\nu(\omega) \in \{p_k \mid \nu \leq k \leq \lceil 2^{\nu+1}\ln(\nu+1)\rceil\}\subset D^\nu$ and
\[
\begin{aligned}
	&\liminf_{\nu \to \infty} \sup_{x \in D^\nu} \inf_{y,\hat{y} \in \ball(x,\delta^\nu)} \Big| \Ex[\nabla_x f(\bm{\xi},y)] -\tfrac{1}{\nu}\nsum_{i=1}^\nu\nabla_x f(\bm{\xi}^i(\omega),\hat{y}) \Big| \\
	\geq\;& \inf_{\nu \geq \bar{\bm{\nu}}(\omega)}  \inf_{y,\hat{y} \in \ball(\bm{p}^\nu(\omega),\delta^\nu)} \Big| \Ex[\nabla_x f(\bm{\xi},y)] -\tfrac{1}{\nu}\nsum_{i=1}^\nu\nabla_x f(\bm{\xi}^i(\omega),\hat{y}) \Big| =|\tfrac{1}{2} - 1| = \tfrac{1}{2},
\end{aligned}
\]
which completes the proof, since $(\Omega,\mathcal{F})$ is $\mathbb{P}$-complete.

\subsection{Proof of \Cref{prop:cvx-1dim}}
	
	Since $f(\xi, \cdot)$ is real-valued and convex, by \cite[Theorem 23.1]{Rockafellar.70}, the directional derivative of $f(\xi, \cdot)$ exists at any point $x$ and in every direction $w$, and we denote it by $f_x'(\xi, x; w)$. 
	Using \cite[Theorem 7.46]{shapiro2021lectures}, the function $x\mapsto\Ex[f(\bm{\xi},x)]$ is convex and directionally differentiable with Borel measurable $f_x'(\cdot, x;w)$ and $
	(\Ex[f(\bm{\xi}, \cdot)])'(x;w)=\Ex[f_x'(\bm{\xi}, x;w)].
	$
	By \cite[Theorem 2.7.2]{clarke1990optimization}, one has $\Ex[\partial_x f(\bm{\xi},x)]=\partial\Ex[ f(\bm{\xi},\cdot)](x)$.
	From \cite[Theorems V.3.3.8, V.3.3.3]{hiriart2013convex}, \cite[Theorem 23.2]{Rockafellar.70}, and the non-emptiness of compact $\partial_x f(\xi,x)$, we can write the Hausdorff distance as follows
	\[
	\setd\Big(\Ex[\partial_x f(\bm{\xi},x)], \tfrac{1}{\nu}\nsum_{i=1}^\nu\partial_x f(\bm{\xi}^i,x)\Big)=\sup_{|w|=1} \Big|\Ex[f_x'(\bm{\xi}, x;w)] - \tfrac{1}{\nu}\nsum_{i=1}^\nu f_x'(\bm{\xi}^i,x;w)\Big|.
	\]
	Let $P=\mathbb{P}\circ\bm{\xi}^{-1}$.
	It suffices to show that, for any fixed $w \in \{-1,1\}$, the class of Borel measurable functions $\mathcal{S}=\{f_x'(\cdot, x;w) \mid x \in \reals\}$
	is $P$-Glivenko--Cantelli; i.e.,
	\begin{equation}\label{proof:prop1-gc}
	\mathbb{P}\left(\lim_{\nu \to \infty}\sup_{x \in \reals} \Big|\Ex[f_x'(\bm{\xi}, x;w)] - \tfrac{1}{\nu}\nsum_{i=1}^\nu f_x'(\bm{\xi}^i,x;w)\Big| = 0 \right) = 1,
	\end{equation}
	which is well defined as explained below.
	For simplicity, assume $w = 1$. Using the fact that $f(\xi,\cdot)$ is $L(\xi)$-Lipschitz, we have $|f_x'(\xi, x; w)| \leq L(\xi)$ for any $x$. Because $\Ex[L(\bm{\xi})] < \infty$, the class $\mathcal{S}$ admits an integrable envelope $L$. 
	From \cite[Theorem 7.42]{shapiro2021lectures} and \cite[Theorem 24.1]{Rockafellar.70}, the function $t\mapsto f_x'(\xi,t; w)$ is non-decreasing and right-continuous (when $w=-1$, it is non-increasing and left-continuous), thereby $\mathcal{S}$ is pointwise measurable \cite[Example 2.3.4]{vanderVaartWellner.23} and hence $P$-measurable \cite[Definition 2.3.3]{vanderVaartWellner.23} with \eqref{proof:prop1-gc} being well defined. With monotonicity and boundedness of $f_x'(\xi,\cdot; w)$, set
	\[
	l(\xi) = \lim_{k\to \infty} f_x'(\xi,-k; w),\quad u(\xi)= \lim_{k\to \infty}f_x'(\xi,k; w),
	\]
	where $l,u:\Xi \to \reals$ are Borel measurable with $\Ex[u(\bm{\xi})]-\Ex[l(\bm{\xi})] \leq 2\Ex[L(\bm{\xi})]$.
	If $\Ex[u(\bm{\xi})]=\Ex[l(\bm{\xi})]$, then $f_x'(\xi,\cdot;w)$ is almost surely constant and the claim is trivial.
	Assume $\Ex[(u-l)(\bm{\xi})]>0$, and fix any $0 < \epsilon < \tfrac{2}{3}\Ex[(u-l)(\bm{\xi})]$. 
	For each $n \in \nats$, define
		\[
	t_{n} = \min \big\{t\in \reals \bigm|\Ex[f_x'(\bm{\xi},t; w)] \geq \min\{\Ex[u(\bm{\xi})]-\tfrac{\epsilon}{2}, \Ex[l(\bm{\xi})] + n\epsilon\}\big\},
	\]
	where the attainment is from the right-continuity of $t\mapsto\Ex[f_x'(\bm{\xi},t; w)]$.
	Let $N\in\nats$ be the first number such that $\Ex[f_x'(\bm{\xi},t_N; w)] \geq \Ex[u(\bm{\xi})]-\tfrac{\epsilon}{2}$.
	By construction, we have $N \leq 1+2\Ex[L(\bm{\xi})]/\epsilon < \infty$ and $\Ex[|u(\bm{\xi}) - f_x'(\bm{\xi},t_N; w)|] \leq \epsilon$.
	By \cite[Theorem 24.1]{Rockafellar.70}, we get $-f_x'(\xi,t_{n}; -w)=\lim_{s \uparrow t_{n}}f_x'(\xi,s; w)$. Using dominated convergence theorem and the definition of $t_n$, for any $n \in \nats$, one has
	\[
	\Ex[-f_x'(\bm{\xi},t_{n}; -w)]=\lim_{s \uparrow t_{n}} \Ex[f_x'(\bm{\xi},s; w)]\leq \Ex[l(\bm{\xi})] + n\epsilon,
	\]
	which yields that $\Ex[|-f_x'(\bm{\xi},t_{n+1}; -w)-f_x'(\bm{\xi},t_{n}; w) |] \leq \epsilon$ whenever $t_{n+1}\neq t_n$ and $\Ex[|-f_x'(\bm{\xi},t_1; -w)-l(\bm{\xi})|] \leq \epsilon$.
	Then, the following (possibly empty when $t_n=t_{n+1}$) sets of functions
	\[
	\begin{aligned}
	\big[l, -f_x'(\cdot,t_{1}; -w) \big] &= \{\phi \in \mathcal{S} \mid l \leq \phi \leq -f_x'(\cdot,t_{1}; -w)\}, \\
	\big[f_x'(\cdot,t_{n}; w), -f_x'(\cdot,t_{n+1}; -w) \big] &= \{\phi \in \mathcal{S} \mid f_x'(\cdot,t_{n}; w) \leq \phi \leq -f_x'(\cdot,t_{n+1}; -w)\},\quad\forall 1 \leq n < N, \\
	\big[f_x'(\cdot,t_{N}; w), u \big] &= \{\phi\in \mathcal{S}\mid f_x'(\cdot,t_{N}; w) \leq \phi \leq u\},
	\end{aligned}
	\]
	are $\epsilon$-brackets of $\mathcal{S}$ in $L_1(P)$-(semi)norm. Meanwhile, for any $x \in \reals$, one has 
	\[
	f_x'(\cdot, x; w) \in  \bigcup\nolimits_{n < N}\big[f_x'(\cdot,t_{n}; w), -f_x'(\cdot,t_{n+1}; -w) \big]\cup \big[l, -f_x'(\cdot,t_{1}; -w) \big]\cup \big[f_x'(\cdot,t_{N}; w), u \big].
	\] 
	Hence, the bracketing number 
	$
	\mathcal{N}_{[\,]}(\epsilon,\mathcal{S},L^1(P))
	$ is finite.
	Invoking \cite[Theorem 2.4.1]{vanderVaartWellner.23}, we conclude that the class $\mathcal{S}$ is $P$-Glivenko--Cantelli. Applying the same argument to $w = -1$, with simple adjustments for nonincreasingness and left-continuity, and then taking the maximum over $w \in \{-1,1\}$ completes the proof.

\subsection{Proof of \Cref{thm:hard-cvx-comp}}\label{sec:proof-cvx}

\paragraph{Intuition.} 
When $f$ is random convex, the ``square-wave'' function $g$ used in \Cref{sec:proof-lip} cannot be extended in a straightforward way, since convex subdifferentials must be monotone. The remedy is to introduce an orthogonal direction and control the activation of the ``pieces'' by a smoothed ``square-wave'' function. This smoothed function must have decreasing magnitude near the origin, due to the assumption about Lipschitz gradients, hence cannot create a constant gap on its own. However, we can exploit the tiny oscillations in its function values near the origin, combined with the nonsmooth map $t \mapsto \max\{t,0\}$ and the orthogonal direction with a constant slope, to amplify the ``signal'' coming from the random bits. The proof again proceeds in five steps.

\paragraph{Step 1.}
Let $\rho:\reals \to [0,1]$ be a twice continuously differentiable function such that 
\[
\supp(\rho)=[-1,1],\quad\rho([-\tfrac{1}{2},\tfrac{1}{2}])=1,\quad\rho'([-\tfrac{1}{2},\tfrac{1}{2}])=0,\quad|\rho'(\cdot)|\leq 30,\quad|\rho''(\cdot)|\leq 30, 
\]
where $\supp(\rho)$ is the \emph{support} of the function $\rho$.
One possible choice is 
	\[
	\rho(t)= \left\{ 
\begin{array}{rl} 
0 & \mbox{if } |t| \geq 1,\\{} 
1-\theta(2|t|-1) & \mbox{if } |t| \in (\tfrac{1}{2},1),\\{} 
1 & \mbox{if } |t| \leq \tfrac{1}{2},
\end{array}\right.
	\]
	where $\theta(t)=6t^5-15t^4+10t^3$.
	For each $k \in \nats$, define $\psi_k:\reals \to [0,1]$ as
	\[
	\psi_k(t)=\eta_k \rho\Big(\frac{t-\tfrac{1}{k}}{r_k}\Big),
	\]
	where $r_k=\frac{1}{8k^2},\eta_k=r_k^2=\frac{1}{64k^4}$.
	Correspondingly, one has 
	\begin{equation}\label{eq:cvx-comp-psi}
	\supp(\psi_k)=\ball(\tfrac{1}{k},r_k), \quad\psi_k(\ball(\tfrac{1}{k},\tfrac{r_k}{2}))=\eta_k, \quad\psi'_k(\ball(\tfrac{1}{k},\tfrac{r_k}{2}))=0, \quad|\psi_k'(\cdot)|\leq 30, \quad|\psi_k''(\cdot)|\leq 30.
	\end{equation}
	Define a random smooth function $g:\Xi\times \reals^2\to \reals$ as
	\[
	g(\xi,x)=  
 	x_1 + \nsum_{k=1}^\infty \psi_k(x_2)(2\bit_k(\xi)-1),
	\]
	which is finite valued as explained below.
Let $h=\max\{\cdot, 0\}$ and  $f:\Xi \times \reals^2\to \reals$ be
\[
f(\xi,x)=h(g(\xi,x)) + 35\|x\|^2= \max\{g(\xi,x), 0 \} + 35\|x\|^2;
\]
see \Cref{fig:cvx}. 
Define $\delta^\nu=(2240\lceil 2^{\nu+1}\ln(\nu+1)\rceil^4)^{-1}.$

\begin{figure}[t]
	\centering
  \includegraphics[width=0.5965\textwidth]{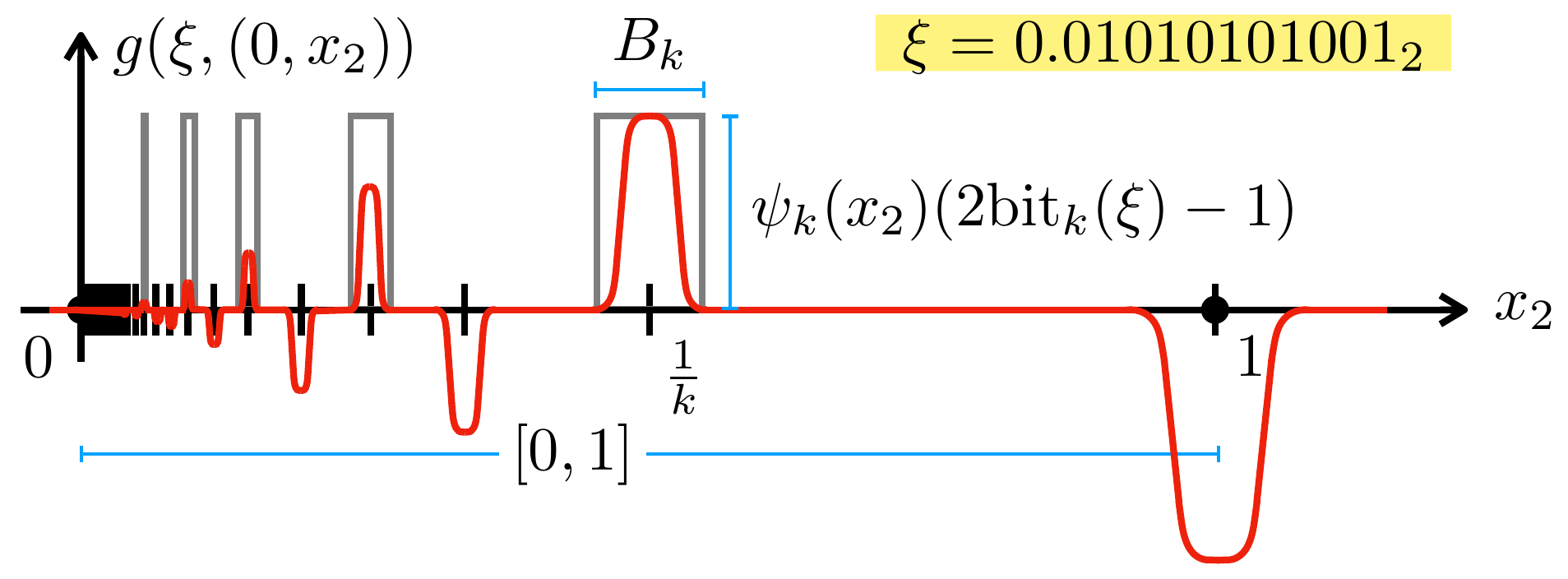}\hspace{.5em}
  \includegraphics[width=0.31\textwidth]{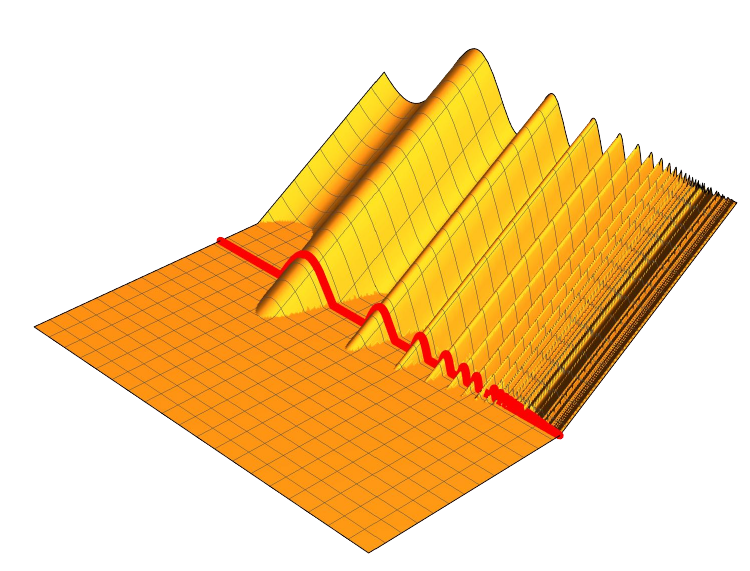}
  \caption{The functions $g(\xi,(0,\cdot))$ and $\max\{g(\xi,\cdot),0\}$ in \Cref{sec:proof-cvx} when $\xi=0.01010101001_2$.}\label{fig:cvx}
\end{figure}

\paragraph{Step 2.}
Let $C=70$.
Since
	$\frac{1}{2}(\frac{1}{k} - \frac{1}{k+1}) > \frac{1}{8k^2}=r_k
	$ for all $k \in \nats$,
	the sets $\{\supp(\psi_k)\}_k$ are disjoint, hence, for any $t \in \reals$, at most one element from $\{\psi_k(t)\}_k$ can be positive. For $x=(x_1,x_2),y=(y_1,y_2)$, at most two  (say, with indices $k_1,k_2$) in $\{\psi_k(x_2)\}_k\cup \{\psi_k(y_2)\}_k$ can be positive, which gives
	\[
	|g(\xi,x)-g(\xi,y)| \leq |x_1-y_1|+\nsum_{j\in\{k_1,k_2\}} |\psi_{j}(x_2) - \psi_{j}(y_2)| \leq 61\|x-y\|.
	\]
Then $g(\xi,\cdot)$ is $C$-Lipschitz continuous for every $\xi \in \Xi$. We next show that $g(\xi,(0,\cdot))$ is differentiable at $0$. If $t \in \supp(\psi_k)$ for some $k$, then $t \ge \frac{1}{2k}$ and
$
|g(\xi,(0,t))| \le \eta_k\le t^2.
$
If $t \notin \cup_{k=1}^\infty \supp(\psi_k)$, then $g(\xi,(0,t))=0$. Hence
$
|t|^{-1}|g(\xi,(0,t))-g(\xi,(0,0))| \le |t| \to 0
$
as $t \to 0$,
so $g(\xi,(0,\cdot))$ is differentiable at $0$ with $g(\xi,(0,\cdot))'(0)=0$. Consequently, a similar argument based on disjoint supports gives
$$
\nabla_x g(\xi,x)=\left[\begin{array}{c}
1 \\
\nsum_{k=1}^\infty \psi_k'(x_2)(2\bit_k(\xi)-1)
\end{array}\right]
$$
and
$
\|\nabla_x g(\xi,x)-\nabla_x g(\xi,y)\|\le C\|x-y\|.
$
Hence $g(\xi,\cdot)$ is $C$-smooth.
Hence, $f(\xi,\cdot)$ is convex and $2C$-Lipschitz on $X$.
From \Cref{prop:bit-measure}, the functions $\xi\mapsto\bit_k(\xi)$, $\xi\mapsto g(\xi,x)$, and then $\xi\mapsto f(\xi,x)$ are measurable.
	Therefore, $f$ and $g$ are Carath\'eodory. Moreover, for any $x \in \reals^2$, one has $|\Ex[f(\bm{\xi},x)]| \leq \Ex[|g(\bm{\xi},x)|] + 35\|x\|^2 \leq |x_1|+\tfrac{1}{64}+35\|x\|^2<\infty.$
	\paragraph{Step 3.} Let $p_k = (0,\tfrac{1}{k}) \in \reals^2$. 
 Then
	for any $k \in \nats$ and $y \in \ball(p_k,\tfrac{r_k}{2})$, by \eqref{eq:cvx-comp-psi}, one has
	\[
	\nsum_{j=1}^\infty \psi_j(y_2) (2\bit_j(\xi)-1)=\eta_k(2\bit_k(\xi)-1), \qquad
	\nsum_{j=1}^\infty \psi_j'(y_2) (2\bit_j(\xi)-1)=0,
	\]
	which implies that 
	\[
	g(\xi,y)= 
	\left\{ 
\begin{array}{rl} 
y_1+\eta_k & \mbox{if } \bit_k(\xi)=1,\\{} 
y_1-\eta_k & \mbox{if } \bit_k(\xi)=0, 
\end{array}\right.\qquad \nabla_x g(\xi,y)=
\left[ \begin{array}{c}
 	1 \\ 0
 \end{array}\right].
	\]
Let $\Delta_k=\frac{1}{32Ck^4}< \frac{1}{16k^2}= \frac{r_k}{2}$.
	For any $\xi \in \Xi, k \in \nats$, and $y \in \ball(p_k, \Delta_k)$, one has $|y_1| \leq \Delta_k <\frac{1}{64k^4}=\eta_k$. Thereby, it follows that
	\begin{equation}\label{eq:cvx-comp-argmax}
	\sgn(g(\xi,y)) = 2\bit_k(\xi) -1,\quad
	\partial h(g(\xi,y)) = \{\bit_k(\xi)\},
	\end{equation} 
	and $f(\xi,\cdot)$ is smooth on $\cup_k \ball(p_k,\Delta_k)$ for any $\xi \in \Xi$. 
	Hence, for any $k \in \nats$ and $y \in \ball(p_k,\Delta_k)$, by \cite[Exercise 10.26]{VaAn}, we can write 
	\[
	\nabla_xf(\xi,y) = \nabla_x g(\xi,y)  \partial h(g(\xi,y)) + 70y=
	\nabla_x g(\xi,y)\bit_k(\xi) + 70y
	= \left[ \begin{array}{c}
 	\bit_k(\xi)  \\ 0
 \end{array}\right] + 70y,
	\]
	where we omit curly braces for singletons for simplicity.
	\paragraph{Step 4.} 
Let $\mathcal{F}$-measurable $E\subset \Omega$ with $\mathbb{P}(E)=1$, $\bar{\bm{\nu}}$, and $\{\bm{k}^\nu\}_\nu$ be defined in \Cref{prop:E}.
	For any $\omega \in E$, $\nu \geq \bar{\bm{\nu}}(\omega)$, one has $\omega \in E^\nu_{\bm{k}^\nu(\omega)}$ for some $\nu \leq \bm{k}^\nu(\omega) \leq \lceil 2^{\nu+1}\ln(\nu+1)\rceil$.
	Then, $\bit_{\bm{k}^\nu(\omega)}(\bm{\xi}^i(\omega))=1$ for all $i \in [\nu]$. Let $\bm{p}^\nu(\omega)=p_{\bm{k}^\nu(\omega)} \in \{p_k\}_{k}$.  For any $y \in \ball(\bm{p}^\nu(\omega), \Delta_{\bm{k}^\nu(\omega)})$, we have
	\[
	\tfrac{1}{\nu}\nsum_{i=1}^\nu\nabla_x f(\bm{\xi}^i(\omega),y)= \left[ \begin{array}{c}
 	\tfrac{1}{\nu}\nsum_{i=1}^\nu  \bit_{\bm{k}^\nu(\omega)}(\bm{\xi}^i(\omega)) \\ 0
 \end{array}\right] + 70y=
		\left[ \begin{array}{c}
 	1 \\ 0
 \end{array}\right] + 70y.
	\]
	Meanwhile, for any $k \in \nats$ and $y \in \ball(p_k, \Delta_k)$, deterministically, we have 
	\[
	\Ex[\nabla_xf(\bm{\xi},y)] = \left[ \begin{array}{c}
 	\Ex[\bit_k(\bm{\xi})] \\ 0
 \end{array}\right] + 70y=
 \left[ \begin{array}{c}
 	1/2 \\ 0
 \end{array}\right] + 70y.
	\]
\paragraph{Step 5.} Note that $C=70$, and
\[
0\leq\delta^\nu 
=
\big(2240\lceil 2^{\nu+1}\ln(\nu+1)\rceil^4\big)^{-1}
= \Delta_{\lceil 2^{\nu+1}\ln(\nu+1)\rceil}
\leq 
\Delta_{\bm{k}^\nu(\omega)}.
\]
Therefore, for any $\omega \in E$,  we conclude that $\bm{p}^\nu(\omega) \in \{p_k \mid \nu \leq k \leq \lceil 2^{\nu+1}\ln(\nu+1)\rceil\}\subset D^\nu$ and
\[
\begin{aligned}
	&\liminf_{\nu \to \infty} \sup_{x \in D^\nu} \inf_{y,\hat{y} \in \ball(x,\delta^\nu)} \Big\| \Ex[\nabla_x f(\bm{\xi},y)] -\tfrac{1}{\nu}\nsum_{i=1}^\nu\nabla_x f(\bm{\xi}^i(\omega),\hat{y}) \Big\| \\
	\geq\;& \sup_{\nu' \geq \bar{\bm{\nu}}(\omega)}  \inf_{\nu \geq \nu'} \inf_{y,\hat{y} \in \ball(\bm{p}^\nu(\omega),\delta^\nu)} \Big\| \Ex[\nabla_x f(\bm{\xi},y)] -\tfrac{1}{\nu}\nsum_{i=1}^\nu\nabla_x f(\bm{\xi}^i(\omega),\hat{y}) \Big\| \geq \tfrac{1}{2} - \inf_{\nu' \geq \bar{\bm{\nu}}(\omega)}  \sup_{\nu \geq \nu'}140\delta^\nu=  \tfrac{1}{2},
\end{aligned}
\]
which completes the proof, since $(\Omega,\mathcal{F})$ is $\mathbb{P}$-complete.

\subsection{Proof of \Cref{cor:hard-ghp,cor:hard-point}}\label{sec:proof-cor}
We use the notation in \Cref{sec:proof-lip}. The proof then follows from the following three observations.
\begin{enumerate}[label=\textnormal{(\alph*)}]
\item Let $h(\xi,x)=f(\xi,x) - x$. For any $\omega \in E$ and large $\nu\in\nats$, the point $\bm{p}^\nu(\omega)$ in \textbf{Step 4} of \Cref{sec:proof-lip} is a local minimizer of $\tfrac{1}{\nu}\nsum_{i=1}^\nu h(\bm{\xi}^i(\omega),\cdot)$ and $\bm{p}^\nu(\omega)\to 0$ as $\nu \to \infty$. \label{item:obs1}
\item For $\omega \in E$ and any $\nu\in\nats$, one has  $1\in \partial\left(\tfrac{1}{\nu}\nsum_{i=1}^\nu f(\bm{\xi}^i(\omega),\cdot)\right)(0)$. \label{item:obs2}
\item $\partial \Ex[f(\bm{\xi}, \cdot)](0) \subset [-\frac{1}{2},\frac{1}{2}]$, hence $0\notin\partial \Ex[h(\bm{\xi}, \cdot)](0) \subset [-\frac{3}{2},-\frac{1}{2}]$. \label{item:obs3}
\end{enumerate}
We now prove these claims. From \textbf{Steps 3} and \textbf{5} in \Cref{sec:proof-lip}, for any $\nu \geq \bar{\bm{\nu}}(\omega)$, we know that $
\tfrac{1}{\nu}\nsum_{i=1}^\nu h(\bm{\xi}^i(\omega), \cdot)$ is constant on $\ball(\bm{p}^\nu(\omega), \delta^\nu)$ with $\delta^\nu > 0$. Moreover, $|\bm{p}^\nu(\omega)|\leq p_\nu= 1/\nu$, hence $\bm{p}^\nu(\omega)\to 0$. This proves \ref{item:obs1}. Next, for any fixed $\nu$, we consider the sequence $\{\bm{p}^{\tau}(\omega)\}_{\tau \geq \max\{\nu, \bar{\bm{\nu}}(\omega)\}}$. By \textbf{Step 4} in \Cref{sec:proof-lip} and $\nu \leq \tau$, one has $\bit_{\bm{k}^\tau(\omega)}(\bm{\xi}^i(\omega))=1$ for all $i \in [\nu]$. Hence, we have 
$\tfrac{1}{\nu}\nsum_{i=1}^\nu \nabla_x f(\bm{\xi}^i(\omega),\bm{p}^\tau(\omega))= 1$. Together with $\bm{p}^\tau(\omega)\to 0$ as $\tau \to \infty$, this implies $1 \in \partial (\tfrac{1}{\nu}\nsum_{i=1}^\nu f(\bm{\xi}^i(\omega),\cdot))(0)$, which establishes \ref{item:obs2}. Finally, using \textbf{Step 2} in \Cref{sec:proof-lip} and Fubini, we obtain $
\Ex [f(\bm{\xi},x)] =  \tfrac{1}{2} \int_0^{\max\{x,0\}}\mathds{1}_{\cup_k B_k}(t)\dd t.
$
Therefore, the map $x\mapsto \Ex[f(\bm{\xi},x)]$ is $\tfrac{1}{2}$-Lipschitz continuous, and thus $\partial \Ex[f(\bm{\xi}, \cdot)](0) \subset [-\frac{1}{2},\frac{1}{2}]$, proving \ref{item:obs3}. 
In summary, \Cref{cor:hard-ghp} follows from \ref{item:obs1} and \ref{item:obs3}, while \Cref{cor:hard-point} follows from \ref{item:obs2} and \ref{item:obs3}.

\subsection{Details of the DC Construction}\label{sec:dc}
Let $\bm{\xi}=(\bm{\xi}_1,\bm{\xi}_2)$ be defined by iid random variables $\bm{\xi}_1,\bm{\xi}_2$ on $(\Omega,\mathcal{F},\mathbb{P})$ with uniform distributions on $\Xi = [0,1]$.
We begin with a new lemma, whose proof is a straightforward modification of \Cref{prop:E}.

\begin{lemma}\label{prop:E-new}
There exist an $\mathcal{F}$-measurable set $E\subset \Omega$ with $\mathbb{P}(E)=1$ and $\nats$-valued random variables $\bar{\bm{\nu}},\bm{k}^1, \bm{k}^2,\ldots$ such that for any $\omega \in E$ and any $\nu \geq \bar{\bm{\nu}}(\omega)$, one has $\nu \leq {\bm{k}^\nu(\omega)} \leq \lceil 4^{\nu+1}\ln(\nu+1)\rceil$, 
\[
\bit_{{\bm{k}^\nu(\omega)}}(\bm{\xi}^1_1(\omega))=\cdots = \bit_{{\bm{k}^\nu(\omega)}}(\bm{\xi}_1^\nu(\omega)) = 1,\quad\text{and}\quad
\bit_{{\bm{k}^\nu(\omega)}}(\bm{\xi}^1_2(\omega))=\cdots = \bit_{{\bm{k}^\nu(\omega)}}(\bm{\xi}_2^\nu(\omega)) = 0.
\]
\end{lemma}
Consider the setting and analysis in \Cref{sec:proof-cvx}, replacing \Cref{prop:E} with \Cref{prop:E-new} wherever appropriate.  Using $g:\Xi \times \reals^2 \to \reals$ in \Cref{sec:proof-cvx}, define a DC function $\ell:\Xi^2 \times \reals^2 \to \reals$ as 
\[
\ell(\xi, x) = \max\{g(\xi_1, x), 0\} + 35\|x\|^2 - \max\{g(\xi_2, x), 0\} - 35\|x\|^2.
\]
Let $h(\xi,x)=\ell(\xi,x) - x_1$.  Parallel to \Cref{sec:proof-cor}, we observe:
\begin{enumerate}[label=\textnormal{(\alph*)}]
\item For any $\omega \in E$ and large $\nu\in\nats$, the point $\bm{p}^\nu(\omega)$ in \textbf{Step 4} of \Cref{sec:proof-cvx} using \Cref{prop:E-new} instead of \Cref{prop:E} is a local minimizer of $\tfrac{1}{\nu}\nsum_{i=1}^\nu h(\bm{\xi}^i(\omega),\cdot)$ and $\bm{p}^{\nu}(\omega)\to 0$ as $\nu \to \infty$. \label{item:dc-obs1}
\item For $\omega \in E$ and any $\nu\in\nats$, one has  $(1,0)\in \partial\left(\tfrac{1}{\nu}\nsum_{i=1}^\nu \ell(\bm{\xi}^i(\omega),\cdot)\right)(0)$. \label{item:dc-obs2}
\item $\partial \Ex[\ell(\bm{\xi}, \cdot)](0) =\{0\}$, hence $0\notin\partial \Ex[h(\bm{\xi}, \cdot)](0) =\{(-1,0)\}$. \label{item:dc-obs3}
\end{enumerate}
We now prove these claims. From \textbf{Step 3} in \Cref{sec:proof-cvx}, for any $\nu \geq \bar{\bm{\nu}}(\omega)$ and $i \in [\nu]$, one has $
g(\bm{\xi}^i_1(\omega), \bm{p}^{\nu}(\omega)) = \bm{p}_1^{\nu}(\omega) + \eta_{\bm{k}^{\nu}(\omega)} > 0$ and 
$g(\bm{\xi}^i_2(\omega), \bm{p}^{\nu}(\omega)) = \bm{p}_1^{\nu}(\omega) - \eta_{\bm{k}^{\nu}(\omega)} < 0.$
Therefore, for any $y$ near $\bm{p}^{\nu}(\omega)$, one has $\tfrac{1}{\nu} \nsum_{i=1}^\nu h(\bm{\xi}^i(\omega),y)=  \eta_{\bm{k}^{\nu}(\omega)},$
which is a constant.
Moreover, $\|\bm{p}^\nu(\omega)\|\leq 1/\nu$, hence $\bm{p}^\nu(\omega)\to 0$. This proves \ref{item:dc-obs1}.
Next, for any fixed $\nu$, we consider the sequence $\{\bm{p}^{\tau}(\omega)\}_{\tau \geq \max\{\nu, \bar{\bm{\nu}}(\omega)\}}$. By \textbf{Step 4} in \Cref{sec:proof-cvx} (using \Cref{prop:E-new}) and $\nu \leq \tau$, for all $i \in [\nu]$, one has $
 \bit_{\bm{k}^\tau(\omega)}(\bm{\xi}_1^i(\omega))=1$ and 
 $\bit_{\bm{k}^\tau(\omega)}(\bm{\xi}_2^i(\omega))=0.$
 Hence, from \textbf{Step 3} in \Cref{sec:proof-cvx}, we have 
$
\tfrac{1}{\nu}\nsum_{i=1}^\nu \nabla_x\ell(\bm{\xi}^i(\omega),\bm{p}^\tau(\omega))= (1,0)$.
Together with $\bm{p}^\tau(\omega)\to 0$ as $\tau \to \infty$, this establishes \ref{item:dc-obs2}. 
 Finally, since $\bm{\xi}_1$ and $\bm{\xi}_2$ have identical distributions, observe that $
\Ex[\ell(\bm{\xi}, x)] \equiv 0$.
Therefore, $\partial \Ex[\ell(\bm{\xi}, \cdot)](0) =\{0\}$, proving \ref{item:dc-obs3}. 

\paragraph{Acknowledgement.} This work is supported in part by the Office of Naval Research under grant N00014-24-1-2492.
{\small
	\bibliography{ref}
\bibliographystyle{plainnat}
}

\end{document}

%% file: fast-draft-macro.tex
\usepackage[pdfencoding=auto, psdextra, colorlinks,citecolor=blue!90!black]{hyperref}       
\usepackage{url}            
\usepackage{booktabs}       
\usepackage{amsfonts}       
\usepackage{nicefrac}       
\usepackage{mathrsfs}
\usepackage{graphicx}
\usepackage{amsmath,amssymb,amsthm}
\usepackage[sans]{dsfont}
\usepackage[square,numbers]{natbib}
\usepackage{cleveref}
\usepackage{algorithm,algorithmic}
\usepackage{thmtools}
\usepackage{thm-restate}

\declaretheorem[name=Theorem]{theorem}
\declaretheorem[name=Lemma]{lemma}
\declaretheorem[name=Proposition]{proposition}

\declaretheorem[name=Corollary]{corollary}

\declaretheorem[name=Remark]{remark}

\usepackage{bm}

\def\sgn{\textnormal{sgn}}

\DeclareMathOperator{\dist}{dist}

\DeclareMathOperator{\exs}{exs}

\DeclareMathOperator{\supp}{supp}

\newcommand{\reals}{\mathbb{R}}

\newcommand{\natnums}{{{\rm l} \kern -.13em {\rm N} }}
\newcommand{\nats}{\mathbb{N}}
\newcommand{\Ex}{\mathbb{E}}
\newcommand{\setd}{{ d \kern -.15em l}}

\newcommand{\nsum}{\mathop{\sum}\nolimits}

\renewcommand{\liminf}{\mathop{\rm liminf}}

\newcommand{\nsup}{\mathop{\rm sup}\nolimits}

\newcommand{\lto}{\,{\lower 1pt\hbox{$\rightarrow$}}\kern -10pt
     \hbox{\raise 4pt \hbox{$\, \scriptstyle l$}}\hskip7pt}
\newcommand{\eto}{\,{\lower 1pt\hbox{$\rightarrow$}}\kern -10pt
     \hbox{\raise 4pt \hbox{$\, \scriptstyle e$}}\hskip7pt}
\newcommand{\hto}{\,{\lower 1pt\hbox{$\rightarrow$}}\kern -11pt
     \hbox{\raise 4pt \hbox{$\, \scriptstyle h$}}\hskip7pt}
\newcommand{\pto}{\,{\lower 1pt\hbox{$\rightarrow$}}\kern -11pt
     \hbox{\raise 4.5pt \hbox{$\, \scriptstyle p$}}\hskip7pt}
\newcommand{\cto}{\,{\lower 1pt\hbox{$\rightarrow$}}\kern -11pt
     \hbox{\raise 4pt \hbox{$\, \scriptstyle c$}}\hskip7pt}
\newcommand{\gto}{\,{\lower 1pt\hbox{$\rightarrow$}}\kern -11pt
     \hbox{\raise 4.5pt \hbox{$\, \scriptstyle g$}}\hskip7pt}
\newcommand{\sto}{\,{\lower 1pt\hbox{$\rightarrow$}}\kern -10pt
     \hbox{\raise 4pt \hbox{$\, \scriptstyle s$}}\hskip7pt}

\def\dd{\textrm{d}}